\documentclass{amsart}

\usepackage{txfonts}
\usepackage{graphicx}%
\usepackage{multirow}%
\usepackage{amsmath,amssymb,amsfonts}%
\usepackage{amsthm}%
\usepackage{mathrsfs}%
\usepackage[title]{appendix}%
\usepackage{xcolor}%
\usepackage{textcomp}%
\usepackage{manyfoot}%
\usepackage{booktabs}%
\usepackage{algorithm}%
\usepackage{algorithmicx}%
\usepackage{algpseudocode}%
\usepackage{listings}%

\usepackage{pdfsync}
\usepackage{enumitem}
\usepackage{hyperref}
\hypersetup{
    colorlinks=true,       
    linkcolor=blue,          
    citecolor=blue,        
    filecolor=blue,      
    urlcolor=blue           
}

\newtheorem{theo}{Theorem}
\newtheorem{prop}{Proposition}
\newtheorem{lemm}{Lemma}

\theoremstyle{thmstyletwo}
\newtheorem{coro}{Corollary}
\newtheorem{question}{Question}

\theoremstyle{remark}
\newtheorem{rema}{\bf Remark}
\newtheorem{example}{\bf Example}

\begin{document}

\title{${\mathbb Z}_{p}^{m}$-actions of type $(d;p,n)$}

\author{Rub\'en A. Hidalgo} 
\address{Departamento de Matem\'atica y Estad\'istica, Universidad de la Frontera, Temuco, 4811230, Chile}
\email{ruben.hidalgo@ufrontera.cl}

\author{Maximiliano Leyton-\'Alvarez}
\address{Instituto de Matem\'atica y F\'isica, Universidad de Talca, Talca, 3460000, Chile}
\email{leyton@inst-mat.utalca.cl}


\begin{abstract}
A ${\mathbb Z}_{p}^{m}$-action of type $(d;p,n)$, where $2 \leq d \leq m \leq n$  are integers, is a pair $(S,N)$ where $S$ is a $d$-dimensional compact complex manifold, $N \cong {\mathbb Z}_{p}^{m}$ is a group of holomorphic automorphisms of $S$ such that the quotient orbifold $S/N$ is the $d$-dimensional projective space ${\mathbb P}^{d}$ whose branch locus consists of $n+1$ hyperplanes in general position, each one of branch order $p$.

If $(d;p,n) \notin \{(2;2,5),(2;4,3)\}$ and $d+1 \leq n$,  then we prove that: (i)  $N$ is a normal subgroup of ${\rm Aut}(S)$ and (ii) if $(S,M)$ is a ${\mathbb Z}_{\hat{p}}^{\hat{m}}$-action of type $(d;\hat{p},\hat{n})$, then $M=N$. If, moreover, $d+1 \leq n \leq 2d-1$, then we observe that $S$ is not algebraically hyperbolic. 
\end{abstract}




\maketitle

\section{Introduction}
Let $S$ be a compact complex manifold of dimension $d \geq 1$.
Its group ${\rm Aut}(S)$ of holomorphic automorphisms is known to be a complex Lie group \cite{BM} and 
there is a natural short exact sequence $1 \to {\rm Aut}^{0}(S) \to {\rm Aut}(S) \to {\rm Aut}(S)/{\rm Aut}^{0}(S)$, where ${\rm Aut}^{0}(S)$ denotes the connected component of the identity. Let $N$ be a subgroup of ${\rm Aut}(S)$ which acts properly discontinuously on $S$; so, we have associated the quotient orbifold $S/N$.
 We are interested in the following two natural questions: 
 \begin{enumerate}
 \item[(1)] May we decide, in terms of the structure of the quotient orbifold $S/N$, if $N$ is a normal subgroup of ${\rm Aut}(S)$? 
 \item[(2)] Let $M$ be another properly discontinuous subgroup of ${\rm Aut}(S)$, which is isomorphic as an abstract group to $N$ and such that the quotient orbifolds $S/N$ and $S/M$ are homeomorphic. May we decide, in terms of the structure of the quotient orbifold, if $N=M$?. 
\end{enumerate}
 
\medskip
 
In this paper, we investigate the above questions in a very particular class of manifolds. More precisely, we consider those pairs $(S,N)$, where 
$N \cong {\mathbb Z}_{p}^{m}$, $m \geq 1$ and $p \geq 2$ are integers, and the quotient orbifold $S/N$ is the $d$-dimensional projective space ${\mathbb P}^{d}$ whose branch locus consists of $n+1$ hyperplanes in general position, each one of branch order $p$. Let us recall that the hyperplanes are in general position if: (i) the intersection of every subcollection of $1 \leq k \leq d$ hyperplanes has dimension $d-k$, and (ii) every subcollection of $k \geq d+1$ hyperplanes has empty intersection.
 In this situation, we will say that $(S,N)$ is a  ${\mathbb Z}_{p}^{m}$-action of type $(d;p,n)$.
Necessarily, $ d \leq m \leq n$, and  $S$ is known to be projective, i.e., it may be holomorphically embedded in some projective space (and ${\rm Aut}(S)$ is a group of biregular automorphisms). If $n=d$, then $S$ is isomorphic to ${\mathbb P}^{d}$. If $n=m=d+1$, then $S$ is isomorphic to the Fermat hypersurface of degree $p$.

\begin{theo}
Let $(S,N)$ is a  ${\mathbb Z}_{p}^{m}$-action of type $(d;p,n) \notin \{(2;2,5),(2;4,3)\}$ and $3 \leq d+1 \leq n$. Then
(i) ${\rm Aut}(S)$ is finite, (ii) $N$ is a normal subgroup of ${\rm Aut}(S)$, and  (iii) if $(S,M)$ is a ${\mathbb Z}_{q}^{r}$-action, then $M=N$.
\end{theo}

We should note that the facts (ii) and (iii), in the previous result, are not generally true for the case of curves (i.e, $d=1$). 
 
\medskip

Examples of compact complex manifolds, for which the group of holomorphic automorphisms is finite, are provided by the so-called algebraically hyperbolic manifolds \cite{BKV}. In \cite{Demailly},  Demailly observed that every compact complex Kobayashi hyperbolic manifold is algebraically hyperbolic. In the same paper, he conjectured the converse. 
 
Now, if $(S,N)$ is a ${\mathbb Z}_{p}^{m}$-action of type $(d;p,n) \notin \{(2;2,5),(2;4,3)\}$, where $3 \leq d+1 \leq n$, then ${\rm Aut}(S)$ is finite. It seems natural to ask if $S$ is algebraically hyperbolic. 
The next result is a negative answer in some cases.

\begin{theo}
Let $(S,N)$ be a ${\mathbb Z}_{p}^{m}$-action of type $(d;p,n) \notin \{(2;2,5),(2;4,3)\}$, where $3 \leq d+1 \leq n$.
If either (i) $n \leq 2d-1$, or (ii) $n=2d$ and $p \in \{2,3\}$,  or (iii) $n=2d+1$ and $p=2$, then $S$ is not algebraically hyperbolic, in particular, not Kobayashi hyperbolic. 
\end{theo}

A natural question is whether the exceptional cases provided in the above result are the only ones for which $S$ is not algebraically hyperbolic.
 
\medskip
 

{\bf Notations:}
Suppose $Y\subset {\mathbb P}^{k}$ is a smooth irreducible projective complex algebraic variety of dimension $d$. In that case, we will denote by $\rm{Aut}(Y)$ its group of all holomorphic automorphisms and by $\rm{Lin}(Y)$ its group of linear automorphisms (that is, its automorphisms obtained as the restriction of a projective linear transformation of ${\mathbb P}^{k}$).

\section{Generalized Fermat varieties}
As noticed above, the maximal value of $m$, in the definition of ${\mathbb Z}_{p}^{m}$-action of type $(d;p,n)$, is $m=n$. 
Also, as observed in \cite{HHL23}, $n \geq d$.

\subsection{The group $H$}
Let $n \geq 1, p \geq 2$ be integers. Set $\omega_{p}=e^{2 \pi i/p}$. 
Let us consider the linear automorphisms 
$\varphi_{1},\ldots,\varphi_{n+1} \in {\rm PGL}_{n+1}({\mathbb C})$  of ${\mathbb P}^{n}$, defined by 
$$\varphi_j([x_1:\cdots:x_j:\cdots:x_{n+1}]):=[x_1:\cdots:\omega_{p}x_j:\cdots:x_{n+1}].$$
Then $\varphi_{1} \circ \cdots \circ \varphi_{n+1}=1$ and
$H:=\langle\varphi_1,\cdots,\varphi_{n}\rangle \cong {\mathbb Z}_{p}^{n}$.
We say that $\{\varphi_{1},\ldots, \varphi_{n+1}\}$ is a set of canonical generators of $H$.

Let us denote by ${\rm Aut}_{g}(H)$ the group of automorphisms of $H \cong {\mathbb Z}_{p}^{n}$ which correspond to permutations of the set of canonical generators  $\{\varphi_{1},\ldots,\varphi_{n+1}\}$. 
Note that ${\rm Aut}_{g}(H)=\langle \Psi_{1}, \Psi_{2}\rangle \cong {\mathfrak S}_{n+1}$, where
$$\Psi_{1}:(\varphi_{1}, \ldots, \varphi_{n+1}) \mapsto (\varphi_{2},\varphi_{1},\varphi_{3},\ldots,\varphi_{n+1}),\;
\Psi_{2}:(\varphi_{1}, \ldots, \varphi_{n+1}) \mapsto (\varphi_{n+1},\varphi_{1},\varphi_{2},\ldots,\varphi_{n}).$$

\subsection{Generalized Fermat pairs}
A generalized Fermat pair of type $(d;k,n)$ is a 
 ${\mathbb Z}_{p}^{n}$-action $(X,H_{X})$ of type $(d;p,n)$. We also say that $X$ 
 is a generalized Fermat variety of type $(d;p,n)$, and that $H_{X}$ is a generalized Fermat group of type $(d;p,n)$. 
 
 If $d=1$, then $X$ is a closed Riemann surface uniformized by the derived subgroup of a Fuchsian group of signature $(0;p,\stackrel{n+1}{\ldots},p)$); we also say that $X$ is 
 a generalized Fermat curve of type $(p,n)$.

\subsection{Case $n=d$}
In this case, we may assume (up to biholomorphisms) that $X={\mathbb P}^{d}$. The group $H$ is a generalized Fermat group of type $(d;p,d)$. This is not the unique generalized Fermat group of such type, but any other is ${\rm PGL}_{d+1}({\mathbb C})$-conjugated to $H$.

\subsection{Case $n=d+1$}
In this case, (up to biholomorphisms) we may assume that $X=F_{p}=\{x_{1}^{p}+ \cdots+x_{d+2}^{p}=0\} \subset {\mathbb P}^{d+1}$, the Fermat hypersurface of degree $p$.
The group $H$ is a generalized Fermat group of type $(d;p,d+1)$.
If (i) $d \geq 2$ and $(d,p) \neq (2,4)$, or (ii) $d=1$ and $p>3$, then $H$ is the unique generalized Fermat group of type $(d;p,d+1)$, and ${\rm Aut}(X)=H \rtimes {\mathfrak S}_{d+2}$, where ${\mathfrak S}_{d+2}$ is the subgroup of ${\rm PGL}_{d+2}({\mathbb C})$ given by permutations of the coordinates.

\subsection{Case $n \geq d+2$}
Next, we recall the algebraic models of $(X,H_{X})$ and the uniqueness results for generalized Fermat groups.

\subsubsection{\bf The parameter space $\Omega_{n,d}$}
Assume $d \geq 1$, and $n \geq d + 2$ are integers.
If $\Lambda=(\lambda_{i,j}) \in {\rm M}_{(n-d-1) \times d}({\mathbb C})$, then we may consider the collection ${\mathcal B}(\Lambda)$ consisting of the following $(n+1)$ hyperplane in ${\mathbb P}^{d}$:
$$\Sigma_{j}=\{[y_{1}:\cdots:y_{d+1}] \in {\mathbb P}^{d}: y_{j}=0\}, \; j=1,\ldots,d+1,$$
$$\Sigma_{d+2}=\{[y_{1}:\cdots:y_{d+1}] \in {\mathbb P}^{d}: y_{1}+\cdots+y_{d+1}=0\},$$
$$\Sigma_{d+2+j}(\Lambda)=\{[y_{1}:\cdots:y_{d+1}] \in {\mathbb P}^{d}: \lambda_{j,1}y_{1}+\cdots+\lambda_{j,d} y_{d}+y_{d+j}=0\}, \; j=1,\ldots,n-d-1.$$

Let us denote by $\Omega_{n,d} \subset {\rm M}_{(n-d-1) \times d}({\mathbb C})$ the subset consisting of those $\Lambda$ such that the above collection is in general position. This space is a connected, open, and dense subset of ${\rm M}_{(n-d-1) \times d}({\mathbb C}) \cong {\mathbb C}^{(n-d-1)d}$.

\subsubsection{\bf A family of algebraic varieties parametrized by $\Omega_{n,d}$}\label{Sec:algebra}
If $\Lambda=(\lambda_{i,j})\in \Omega_{n,d}$, then we may consider the following algebraic variety

\begin{equation}\label{Equation Algebraic Model}
X_{n}^{p}(\Lambda):=\left\{\begin{matrix}
x_{1}^{p}+\cdots+x_{d}^{p}+x_{d+1}^{p}+x_{d+2}^{p}&=&\ 0\\
\lambda_{1,1}x_{1}^{p}+\cdots+\lambda_{1,d}x_{d}^{p}+x_{d+1}^{p}+x_{d+3}^{p}&=&\ 0\\
\vdots&\vdots&\ \vdots\\
\lambda_{n-d-1,1}x_{1}^{p}+\cdots+\lambda_{n-d-1,d}x_{d}^{p}+x_{d+1}^{p}+x_{n+1}^{p}&=&\ 0
\end{matrix}\right\}\subset{\mathbb P}^{n}.
\end{equation}

\begin{rema} \label{rem:sc}
The variety $X^{p}_{n}(\Lambda)$ is an irreducible nonsingular complete intersection projective variety of dimension $d$. So, 
if $d \geq 2$, then $X^{d}_{n}(\Lambda)$ is simply connected (this result is attributed to Lefschetz; see \cite{Har74}). 
\end{rema}

The following facts can be deduced from the above algebraic model of $X_{n}^{p}(\Lambda)$ and the form of the elements $\varphi_{i}$.
\begin{enumerate}
\item[(I)] ${\mathbb Z}_{p}^{n}\cong H<{\rm Aut}(X^{p}_{n}(\Lambda))<{\rm PGL}_{n+1}({\mathbb C})$.
\item[(II)] $\varphi_{1} \varphi_{2} \cdots \varphi_{n+1}=1$.
\item[(III)] The only non-trivial elements of $H$ with fixed set points being of maximal dimension $d-1$ are the non-trivial powers
of the generators $\varphi_{1}, \ldots, \varphi_{n+1}$. Moreover, for $d \geq 2$, ${\rm Fix}(\varphi_{j}):=\{x_{j}=0\} \cap X^{p}_{n}(\Lambda)$
is isomorphic to a generalized Fermat variety of type $(d-1;k,n-1)$.
\item[(IV)]  $\pi:X^{p}_{n}(\Lambda) \to {\mathbb P}^{d}: [x_{1}:\cdots:x_{n+1}] \mapsto [x_{1}^{p}: \cdots: x_{d+1}^{p}]$ is a Galois branched cover with deck group $H$, whose branch locus is the collection ${\mathcal B}(\Lambda)$. In particular, $(X_{n}^{p}(\Lambda),H)$ is a generalized Fermat pair of type $(d;p,n)$
\end{enumerate}

\begin{rema}
As a consequence of Randell's isotopy theorem \cite{Randell}, for $\Lambda_{1}, \Lambda_{2} \in \Omega_{n,d}$, there is an orientation-preserving homeomorphism 
 $f:{\mathbb P}^{d} \to {\mathbb P}^{d}$ carrying ${\mathcal B}(\Lambda_{1})$ onto ${\mathcal B}(\Lambda_{2})$. This homeomorphism lifts to 
 an orientation-preserving homeomorphism 
 $h:X_{n}^{p}(\Lambda_{1}) \to X_{n}^{p}(\Lambda_{2})$ such that $h H h^{-1}=H$.
\end{rema}

The following fact was obtained in \cite{HHL23}, as a consequence of the results in \cite{Kon02,MaMo64}.

\begin{theo}[\cite{HHL23}]\label{maintheo2}
(1) The linear group ${\rm Lin}(X^{p}_{n}(\Lambda))$ consists of matrices such that only an element in each row and column is non-zero.
(2) If $(d;p,n) \notin \{(2;2,5),(2;4,3)\}$, then ${\rm Aut}(X^{p}_{n}(\Lambda))={\rm Lin}(X^{p}_{n}(\Lambda))$.
\end{theo}

\subsubsection{\bf Algebraic equations of all generalized Fermat varieties}
Let $(X,H_{X})$ be a generalized Fermat pair of type $(d;p,n)$ and let $\pi:X \to {\mathbb P}^{d}$ be a 
Galois branched cover, with deck group $H_{X}$, and whose branch locus consists of $(n+1)$ hyperplanes $B_{1},\ldots,B_{n+1}$ which are in general poistion. Let us consider any permutation $\sigma \in {\mathfrak S}_{n+1}$. There is a unique $T_{\sigma} \in {\rm PGL}_{d+1}({\mathbb C})$ such that $T_{\sigma}(B_{\sigma^{-1}(i)})=\Sigma_{i}$, for $i=1,\ldots,d+2$.
 As the $T_{\sigma}$-image of these $(n+1)$ hyperplanes are in general position, there is a unique $\Lambda=\Lambda_{\sigma} \in \Omega_{n,d}$ such that
$T_{\sigma}(B_{\sigma^{-1}(d+2+j)})=\Sigma_{d+1+j}(\Lambda)$, for $j=1,\ldots,n-1-d$.

\begin{rema}
The above construction of $T_{\sigma} \in {\rm PGL}_{d+1}({\mathbb C})$, for each $\sigma \in {\mathfrak S}_{n+1}$, induces a one-to-one homomorphism $\Theta:{\mathfrak S}_{n+1} \to  {\rm Aut}(\Omega_{n,d})$. We set ${\mathbb G}_{n,d}=\Theta({\mathfrak S}_{n+1}) \cong {\mathfrak S}_{n+1}$.
\end{rema}

\begin{theo}[\cite{GHL09}, \cite{HHL23}]
If $n \geq d+2$ and $(X,H_{X})$ is a generalized Fermat pair of type $(d;p,n)$, then there is some $\Lambda \in \Omega_{n,d}$ and a biholomorphism $\phi:X \to X_{n}^{p}(\Lambda)$ such that $\phi H_{X} \phi^{-1}=H$. Moreover, $\Lambda_{1}, \Lambda_{2} \in \Omega_{n,d}$ produce isomorphic pairs if and only if they belong to the same ${\mathbb G}_{n,d}$-orbit.
\end{theo}

\begin{rema}
The above result, for $d \geq 2$, may be seen as a consequence of Pardini's classification of abelian branched covers \cite{Par91}, and that of maximal branched abelian covers \cite{AlPa13}. The proof of the case $d=1$ in \cite{GHL09} was obtained from Fuchsian group theory.
\end{rema}

\subsection{A simple remark on the cohomological information of generalized Fermat varieties}\label{re:cohom}
The fact that $X:=X^{p}_{n}(\Lambda)$ is a complete intersection variety allows us to compute the cohomology groups of the twisting sheaf $\mathcal{O}_{X}(r)$ in a relatively direct way, and in particular, to obtain the following.

\begin{prop}\label{cohomologia}
Let $d \geq 2$, $\Lambda \in \Omega_{n,d}$, $n \geq d+1$, and $X:=X^{p}_{n}(\Lambda)$. 
Set $r_1=(n-d)p-n-1$. Then
\begin{enumerate}[leftmargin=*,align=left]
\item The plurigenera $P_{m}(X)$ of $X$ satisfies
{\small
$$P_{m}(X)= \frac{p^{n-d}((n-d)p-n-1)^{d}}{d!}m^{d}+O(m^{d-1}).$$
}

\item The arithmetic genus $p_{a}(X)$ and the geometric genus $p_{g}(X)$ are given by
{\small
$$p_a(X)=p_g(X)= \left \{ \begin{array}{ccc}
0 & \mbox{if} & r_1<0\\
            \binom{r_1+n}{n} & \mbox{if} & 0\leq r_1<p\\
            \sum_{j\in \Delta_{r_1}}\binom{r_1-\overline{j}+d}{d} & \mbox{if} & r_1\geq p\\
           \end{array} \right .$$
           }

\item If $(n-d)p-n-1=0$, then $X$ is a Calabi-Yau variety.
\item If  $d=2$, then $X$ is a general type surface except for the rational varieties cases $(p,n)\in \{(2,3), (3,3), (2,4)\}$ and the $K3$ varieties $(p,n)\in \{(4,3), (2,5)\}$.

\end{enumerate}
\end{prop}
\begin{proof}
Let $\mathbb{C}[x_1,...,x_{m}]_l$ be  the  homogeneous polynomials of degree $l$.

\begin{enumerate} [label=(\alph*),leftmargin=*,align=left]
\item We first proceed to describe the cohomology groups of the twisting sheaf   $\mathcal{O}_{X}(r), r\in \mathbb{Z}.$
\begin{enumerate}[leftmargin=*,align=left]
 \item[(a1)]  Let  $\Delta_r:=\{(j_1,...,j_{n-d})\in \mathbb{Z}^{n-d}:\; 0\leq j_i\leq p-1, 0\leq i\leq n-d, \;\mbox{and}\; \overline{j}:=j_1+j_2+\cdots j_{n-d}\leq r \}$. Then
{\small
 $$\displaystyle H^{0}(X, \mathcal{O}_{X}(r)):=\left \{  \begin{array}{ccc}
                                             0 & \mbox{if} & r<0\\
                                           \mathbb{C}[x_1,...,x_{n+1}]_r &\mbox{if} & 0\leq r<p \\
                                           \bigoplus_{j\in \Delta_r}Q_j&\mbox{if}  & r\geq p
                                          \end{array} \right .$$
                                          }
where  $ Q_j:=\mathbb{C}[x_1,...,x_{d+1}]_{(r-\overline{j})}x_{d+2}^{j_1} x_{d+3}^{j_2}\cdots x_{n+1}^{j_{n-d}}$, $j:=(j_1,....j_{n-d})$.
\item[(a2)] By Grothendieck's vanishing theorem, 
 $H^{i}(X, \mathcal{O}_{X}(r))=0$  for $i>d$, and $r\in \mathbb{Z}$,

\item[(a3)] and, as $X$ is a complete intersection variety, 
  $H^{i}(X, \mathcal{O}_{X}(r))=0$ for $0<i<d$, and $r\in \mathbb{Z}$
 (see page 231 of \cite{Har77}).
 
 \item[(a4)]  Finally, using the Serre duality, 
     $H^{d}(X, \mathcal{O}_{X}(r))\cong H^{0}(X, \mathcal{O}_{X}(r_1-r))$.

Remember that $\omega_X\cong  \mathcal{O}_X(r_1)$ (see page 188 of \cite{Har77}).

\end{enumerate}

\item With the former, we can calculate the plurigenus of $X$
{\small
$$P_m(X)=\dim_{\mathbb{C}} H^{0}(X,\omega_X^{\otimes m})=\dim_{\mathbb{C}} H^{0}(X,\mathcal{O}_X(r_m))$$
}
where $r_m:=mr_1=m((n-d)p-n-1)$.
\begin{enumerate}
 \item[(b1)]  If  $ (n-d)p-n-1<0$, we obtain that $P_{m}(X)=0$. This implies that the Kodaira dimension of $X$ is  $\kappa(X)=-\infty$.
 \item[(b2)] If  $(n-d)p-n-1=0$, we obtain that $P_{m}(X)=1$. This implies that the Kodaira dimension of $X$ is   $\kappa(X)=0$.
 \item[(b3)] If $(n-d)p-n-1>0$, the canonical sheaf in very ample and
{\small
 \begin{center}
   $P_m(X)=\left \{ \begin{array}{ccc}

            \binom{r_m+n}{n} & \mbox{if} &0\leq  r_m<p\\
             & &\\
             \sum_{j\in \Delta_{r_m}}\binom{r_m-\overline{j}+d}{d} & \mbox{if} & r_m\geq p\\\
           \end{array} \right .
$
 \end{center}
 }
In particular, if $r_m\geq \max\{ p, (n-d)(p-1)\}$, we obtain the assertion (1).
\end{enumerate}
This implies that the Kodaira dimension of  $X$ is  $\kappa(X)=d$.

\item The former also permits us to determine the arithmetic genus and geometric genus of  $X$. As seen from the above,
$p_a(X)=p_g(X)=\dim_{\mathbb{C}}H^{d}(X,\mathcal{O}_X)=\dim_{\mathbb{C}}H^{0}(X,\mathcal{O}_X(r_1)),$
so, we obtain assertion (2).
 \end{enumerate}

\end{proof}

\subsubsection{\bf Uniqueness of generalized Fermat groups}
If $n=d$, then the generalized Fermat group is not unique (but it is unique up to conjugation).

\begin{theo}[\cite{HKLP17}]
If $d=1$ and $(n-1)(p-1)>2$, then a generalized Fermat curve of type $(p,n)$ has a unique generalized Fermat group.
\end{theo}

\begin{theo}[\cite{HHL23}]\label{teo2}
Let $d \geq 2$ and $(X,H_{X})$ be a generalized Fermat pair of type $(d;p,n) \notin \{(2;2,5),(2;4,3)\}$. If $\hat{H}$ is a generalized Fermat group of $X$ of some type $(d;\hat{p},\hat{n})$, then $\hat{H}=H_{X}$.
\end{theo}
\begin{proof}
We may assume $X=X_{n}^{p}(\Lambda)$, for some $\Lambda \in \Omega_{n,d}$ and $H_{X}=H$. 

Let $\psi \in \hat{H}$ be an element whose fixed point locus has dimension $d-1$ (i.e., a canonical generator for $\hat{H}$). By Theorem \ref{maintheo2}, $\psi \in {\rm Lin}(X)$ corresponds to a matrix such that only an element in each row and column is non-zero. If such a matrix is not diagonal, then its locus of fixed points in ${\mathbb P}^{n}$ is a linear subspace of codimension at least two; so ${\rm Fix}(\psi) \cap X$ cannot have dimension $d-1$, a contradiction. So, 
$$\psi([x_{1}:\cdots:x_{n+1}])=[\alpha_{1} x_{1}: \cdots : \alpha_{n+1} x_{n+1}].$$

If $[x_{1}: \cdots:x_{n+1}] \in X$, then as $\psi \in {\rm Aut}(X)$, it follows that
\begin{equation}
\left\{\begin{matrix}
\alpha_{1}^{p}x_{1}^{p}+\cdots+\alpha_{d}^{p}x_{d}^{p}+\alpha_{d+1}^{p}x_{d+1}^{p}+\alpha_{d+2}^{p}x_{d+2}^{p}&=&\ 0\\
\lambda_{1,1}\alpha_{1}^{p}x_{1}^{p}+\cdots+\lambda_{1,d}\alpha_{d}^{p}x_{d}^{p}+\alpha_{d+1}^{p}x_{d+1}^{p}+\alpha_{d+3}^{p}x_{d+3}^{p}&=&\ 0\\
\vdots&\vdots&\ \vdots\\
\lambda_{n-d-1,1}\alpha_{1}^{p}x_{1}^{p}+\cdots+\lambda_{n-d-1,d}\alpha_{d}^{p}x_{d}^{p}+\alpha_{d+1}^{p}x_{d+1}^{p}+\alpha_{n+1}^{p}x_{n+1}^{p}&=&\ 0
\end{matrix}\right\}\subset{\mathbb P}^{n}.
\end{equation}

Since 
$x_{1}^{p}+\cdots+x_{d}^{p}+x_{d+1}^{p}+x_{d+2}^{p}= 0,$
we may observe that 
$\alpha_{1}^{p}=\cdots=\alpha_{d+1}^{p}=\alpha_{d+2}^{p}.$

Since, for $i=1,\ldots,n-d-1$,
$\lambda_{i,1}x_{1}^{p}+\cdots+\lambda_{i,d}x_{d}^{p}+x_{d+1}^{p}+x_{d+2+i}^{p}=0,$
we also observe that 
$\alpha_{1}^{p}=\cdots=\alpha_{d+1}^{p}=\alpha_{d+2+i}^{p}.$

All of the above asserts that $\psi \in H$ and that it has a $(d-1)$-dimensional locus of fixed points. So, $\psi$ is a non-trivial power of one of the canonical generators of $H$. 

The above asserts that $\hat{H} \leq H$.
Now, by interchanging the roles of $\hat{H}$ and $H$ in the above, we also obtain that $H \leq \hat{H}$.
\end{proof}

\begin{rema}
The two exceptional cases $(d;p,n) \in \{(2;2,5),(2;4,3)\}$ correspond to the only K3-surfaces among generalized Fermat surfaces. They have infinite group of holomorphic automorphisms, the corresponding linear subgroup has infinite index and it is non-normal. Anyway, inside the linear subgroup of automorphisms there is a unique generalized Fermat group.
\end{rema}

\subsection{Automorphisms of generalized Fermat varieties}\label{Ssec:Aut}

As a consequence of Theorem \ref{teo2}, is the following fact, which together with Theorem \ref{maintheo2} below, might be used to explicitly compute the full group of automorphismsm of a generalized Fermat variety.

\begin{coro}\label{coro-unico}
Let $d \geq 2$, $p \geq 2$, $n \geq d+1$ be integers and $(d;p,n) \notin \{(2;2,5), (2;4,3)\}$. Let $(X,H)$ be a generalized Fermat pair of type $(d;p,n)$. If $G_{0}$ is the ${\rm PGL}_{d+1}({\mathbb C})$-stabilizer of the $n+1$ branch hyperplanes of $X/H={\mathbb P}^{d}$, then $|{\rm Aut}(X)|= |G_0|p^{n}$ and, if the order of $G_{0}$ is relatively prime with $p$, then 
${\rm Aut}(X) \cong H \rtimes G_{0}$.
\end{coro}
\begin{proof}
We know that $X$ admits a unique generalized Fermat group $H$ of type $(d;p,n)$. 
Let $\pi:X \to {\mathbb P}^{d}$ be a Galois branched covering, with $H$ as its desk group, and let $\{L_{1},\ldots, L_{n+1}\}$ be its set of branch hyperplanes. Let $G_{0}$ be the ${\rm PGL}_{d+1}({\mathbb C})$-stabilizer of these $n+1$ branch hyperplanes.
As $H$ is a normal subgroup of ${\rm Aut}(X)$, it follows the existence of a homomorphism $\theta:{\rm Aut}(X) \to G_{0}$, with kernel $H$.
As $X$ is a universal branched cover, every element $Q$ of $G_{0}$ lifts to a holomorphic automorphism $\widehat{Q}$ of $X$.
Then there is a short exact sequence
$1 \rightarrow H \rightarrow {\rm Aut}(X) \stackrel{\rho}{\rightarrow} G_0 \rightarrow 1.$
In particular,  $|{\rm Aut}(X)|= |G_0|p^{n}$. Also, by the Schur-Zassenhaus theorem \cite{Dummit}, in the case that the order of $G_{0}$ is relatively prime with $p$, then 
${\rm Aut}(X) \cong H \rtimes G_{0}$.
\end{proof}

\begin{coro}
Let $d \geq 2$ and $p \geq 2$ be integers. If
$G_{0}$ be a finite subgroup of ${\rm PGL}_{d+1}({\mathbb C})$, then there exists a generalized Fermat pair $(X,H)$ of type $(d;p,n)$, for some
$n \geq d+1$, such that ${\rm Aut}(X/H) \cong G_{0}$. In fact, for $|G_{0}| \leq d+1$ we may assume $n=d+1$ and, for $|G_{0}| \geq d+2$, we may assume $n=|G_{0}|-1$.
\end{coro}
\begin{proof}
If $|G_{0}| \leq d+1$, then take $n=d+1$  and note that for the classical Fermat hypersurface $F_{p} \subset {\mathbb P}^{n}$ of degree $p$ one has that ${\rm Aut}(F_{p})/H$ contains the permutation group of $d+1$ letters.
Let us assume $|G_{0}| \geq d+2$.
The linear group $G_{0}$ induces a linear action on the space ${\mathbb P}^{d}_{hyper}$ of hyperplanes of ${\mathbb P}^{d}$. As $G_{0}$ is finite, we may find (generically) a point $q\in {\mathbb P}^{d}_{hyper}$ whose $G_{0}$-orbit is a generic set of points. Such an orbit determines a collection of $|G_{0}|$ lines in general position in ${\mathbb P}^{d}$. Let us observe that, by the generic choice, we may even assume the above set of points to have ${\rm PGL}_{d+1}({\mathbb C})$-stabilizer exactly $G_{0}$, so the same situation for our collection of hyperplanes.
Now, the results follow from Corollary \ref{coro-unico}.
\end{proof}

\subsection{\bf Fixed points of elements of $H$}
Let us consider a generalized Fermat pair $(X_{p}^{n}(\Lambda),H)$ of type $(d;p,n)$, where $d \geq 2$, and let $\pi:X_{n}^{p}(\Lambda) \to {\mathbb P}^{d}$ be as previously defined in Section \ref{Sec:algebra}.
The branch locus of $\pi$ is the collection ${\mathcal B}(\Lambda)$, the union of the following $n+1$ hyperplanes (in general position)
$$\Sigma_{1},\ldots,\Sigma_{d+2},\Sigma_{d+3}=\Sigma_{d+3}(\Lambda),\ldots,\Sigma_{n+1}=\Sigma_{n+1}(\Lambda).$$

Next, we describe those elements of $H$ acting with fixed points on $X_{n}^{p}(\Lambda)$.

\begin{prop}\label{puntosfijos}
Let $\varphi \in H$ be different from the identity. Then $\varphi$ has fixed points on $X_{n}^{p}(\Lambda)$ if and only if there exist 
$1 \leq j \leq d$, $1 \leq i_{1}<\ldots <i_{j} \leq n+1$, and $1 \leq m_{i_{1}},\ldots, m_{i_{j}} \leq p-1$, such that
$\varphi:=\varphi_{i_{1}}^{m_{_{1}}} \circ \cdots \circ \varphi_{i_{j}}^{m_{i_{j}}}$.
\end{prop}
\begin{proof}
Let $p \in X_{n}^{p}(\Lambda)$ be a fixed point of $\varphi$. Then $\pi(p) \in {\mathcal B}(\Lambda)$. Let $1 \leq i_{1}<\ldots <i_{j} \leq n+1$ a maximal collection of indices so that
$p \in \Sigma_{i_{1}} \cap \cdots \cap \Sigma_{i_{j}}$. As the hyperplanes $\Sigma_{j}$ are in general position, necessarily $j \leq d$. Now, the previous asserts that 
$p \in {\rm Fix}(\varphi_{i_{1}}) \cap \cdots \cap {\rm Fix}(\varphi_{i_{j}})$, so $\varphi \in \langle \varphi_{i_{1}},\ldots,\varphi_{i_{j}}\rangle$. The converse is clear.
\end{proof}

\begin{rema}\label{observafijos}
Let $d \geq 2$, $n \geq d+1$, $p \geq 2$, $\Lambda \in \Omega_{n,d}$, $X^{p}_{n}(\Lambda)$.  
Let us consider an element $\varphi \in H$, different from the identity, acting with fixed points on $X^{p}_{n}(\Lambda)$. 
As seen above, we can write
$\varphi:=\varphi_{1}^{m_{1}} \circ \cdots \circ \varphi_{n+1}^{m_{n+1}} \in H$, where there are $1 \leq j \leq d$ and 
$1 \leq i_{1}<\ldots <i_{j} \leq n+1$ such that (i) $m_{i}=0$ if and only if $i \notin \{i_{1},\ldots,i_{j}\}$ and 
(ii) $m_{i_{1}},\ldots,m_{i_{j}} \in \{1,\ldots,p-1\}$. 
For each $l \in \{0,1,\ldots,p-1\}$, set
$$L_{l}(\varphi):=\{j \in \{1,\ldots, n+1\}: m_{j}=l\},$$
and the (possibly empty) algebraic sets
$$\widetilde{F}_{l}(\varphi)=\{[x_{1}:\cdots:x_{n+1}] \in {\mathbb P}^{n} : x_{i}=0, \;  \forall i \notin L_{l}(\varphi)\}, \;
F_{l}(\varphi):=\widetilde{F}_{l}(\varphi)\cap X^{p}_{n}(\Lambda).$$

The locus of fixed points of $\varphi$ in ${\mathbb P}^{n}$ is the disjoint union of the algebraic sets
$\widetilde{F}_{l}(\varphi)$.

Note that each $\widetilde{F}_{l}(\varphi)$ is:
(i) just a point if $\# L_{l}(\varphi)= 1$, and (ii) a projective linear space of dimension $\# L_{l}(\varphi) -1$ if $\# L_{l}(\varphi)> 1$.
The locus of fixed points of $\varphi$ on $X^{p}_{n}(\Lambda)$ is then given as the disjoint union of the sets $F_{l}(\varphi)=\widetilde{F}_{l}(\varphi) \cap X^{p}_{n}(\Lambda)$.
But on $X^{p}_{n}(\Lambda)$ we cannot have points $[x_{1}:\cdots:x_{n+1}]$ with at least $d+1$ coordinates equal to zero. This fact asserts that for $\# L_{l}(\varphi)\leq n-d$ one has that  $F_{l}(\varphi)=\emptyset$.
Also, for
$\# L_{l}(\varphi) \geq n+1-d$, we obtain that $F_{l}(\varphi) \neq \emptyset$ is a generalized Fermat variety of dimension $\# L_{l}(\varphi)+d-n-1$.

In particular, its number of (non-empty) connected components (if non-empty) equals the number of exponents $l$ appearing in $\varphi$ at least $n+1-d$ times.
\end{rema}

\begin{example}
Let $d \geq 2$, $n \geq d+1$, $p \geq 2$, $\Lambda \in \Omega_{n,d}$, $X:=X^{p}_{n}(\Lambda)$.  
\begin{enumerate}[leftmargin=*,align=left]
\item If $p=2$, and $\varphi \in H \cong {\mathbb Z}_{2}^{n}$, different from the identity.
In this case, we have only two sets to consider, say $\#L_{0}(\varphi)$ and $\#L_{1}(\varphi)$, satisfying that $\#L_{0}(\varphi)+\#L_{1}(\varphi)=n+1$.
By Proposition \ref{observafijos}, $\varphi$ has no fixed points on $X^{2}_{n}(\Lambda)$ if and only if
$$\#L_{0}(\varphi),\#L_{1}(\varphi) \leq n-d.$$

Since, $n+1=\#L_{0}(\varphi)+\#L_{1}(\varphi)\leq (n-d)+(n-d)$, necessarily $n \geq 2d+1$. In other words, if $n \leq 2d$, then $H$ does not have non-trivial elements acting freely.

\item
If $d=2$, and $\varphi \in H$, different from the identity. By Proposition \ref{puntosfijos}, ${\rm Fix}(\varphi) \neq \emptyset$ if and only if there exists some $l \in \{0,1,\ldots,p-1\}$ such that 
$\# L_{l}(\varphi) \geq n-1$. In other words, if and only if $\varphi$ is one of the following elements: $\varphi_{i}^{l}$ or $\varphi_{j}^{s} \circ \varphi_{k}^{r}$, where $l,r,s \in \{1,\ldots,p-1\}$, and $i,j,k \in \{1,\ldots,n+1\}$ with $j \neq k$.

\item
Let us assume $p \geq 2$ is a prime integer. Let $K \cong {\mathbb Z}_{p}^{n-r}$ be a subgroup of $H$ acting freely on $X$. Let $F_{j} \subset X$, $j=1,\ldots, n+1$, be the locus of fixed points of the canonical generator $\varphi_{j}$. As $H$ is an abelian group, each $F_{j}$ is invariant under $K$ and acts freely on it.  Let $S=X/K$ (which is a compact complex manifold of dimension $d$) and $X_{j}=F_{j}/K$ (a connected complex submanifold of $S$). The $(n+1)$ connected sets $X_{j}$ are the locus of fixed points of the induced holomorphic automorphism by $\varphi_{j}$. As each two different $F_{i}$ and $F_{j}$ always intersect transversely, it follows that the same happens for $X_{i}$ and $X_{j}$. As the locus of fixed points of (finite) holomorphic automorphisms is smooth, it follows that different $X_{i}$ and $X_{j}$ are the fixed points of different cyclic groups of $N=H/K \cong {\mathbb Z}_{p}^{r}$. This in particular asserts that $n+1 \leq (p^{r}-1)/(p-1)$. So, for instance, the cases (i) $r=1$ and (ii) $r=2$ and $p=2$, are impossible (note that this is in contrast to the case $p=2$ and $d=1$, where these subgroups exist and are related to hyperelliptic Riemann surfaces).

\item Let $n=p=3$ and $d=2$. In this case, $X$ is just the Fermat hypersurface $\{x_{1}^{3}+x_{2}^{3}+x_{3}^{3}+x_{4}^{3}=0\} \subset {\mathbb P}^{3}$. If
$\varphi=\varphi_{1}\varphi_{2}\varphi_{3}^{2}$, then
$(m_{1},m_{2},m_{3},m_{4})=(1,1,2,0)$ and
$L_{0}(\varphi)=\{4\}, \; L_{1}(\varphi)=\{1,2\}, \; L_{2}(\varphi)=\{3\}$.
The locus of fixed points (in ${\mathbb P}^{3}$) of $\varphi$ is given by
{\small
$$\widetilde{F}_{0}(\varphi) \cup \widetilde{F}_{1}(\varphi) \cup \widetilde{F}_{2}(\varphi)=$$
$$\{[0:0:0:1]\} \cup \{[x_{1}:x_{2}:0:0] \in {\mathbb P}^{3}\} \cup \{[0:0:1:0]\}.$$
}
As the cardinalities of $L_{0}(\varphi)$ and $L_{2}(\varphi)$ are at most equal to $n-d$, these two do not introduce fixed points of $\varphi$ on $X$ (this can be seen also directly). The set $L_{1}(\varphi)$ has cardinality $2 \geq n-d+1$, so it produces a zero-dimensional set of fixed points consisting of the three points $[1:-1:0]$, $[1:\omega_{6}:0]$ and $[1:\omega_{6}^{-1}:0]$, where $\omega_{6}=e^{\pi i/3}$.

\item Let us consider the case $n=d+1$, that is, $X$ is the Fermat hypersurface of degree $p$. Let us consider an element $\varphi \in H$, different from the identity. Let us write
$$\varphi=\varphi_{1}^{m_{1}}\circ \cdots \circ \varphi_{d+1}^{m_{d+1}}, \; 0 \leq m_{i} \leq p-1.$$

By Proposition \ref{puntosfijos}, for $\phi$ to act freely on $X$, necessarilly $1 \leq m_{i} \leq p-1$. Since $\varphi_{1} \circ \cdots \circ \varphi_{d+2}=1$, we also have that, for every $i \in \{1,\ldots,d+1\}$, 
$$\varphi=\varphi_{1}^{m_{1}-m_{i}}\circ \cdots \circ \varphi_{i-1}^{m_{i-1}-m_{i}} \circ \varphi_{i+1}^{m_{i+1}-m_{i}}   \cdots \circ \varphi_{d+1}^{m_{d+1}-m_{i}} \circ \varphi_{d+2}^{-m_{i}}.$$
So, for $\varphi$ to acts freely, we must also have that $m_{j}-m_{i} \nequiv 0 \mod(p)$, for every $i \neq j$.

These conditions ensure that the existence of such $\varphi$ obligates for $p \geq d+2$. Now, if $p \geq d+2$, then we may consider $m_{i}=i$, for $i=1,\ldots,d+1$, and set $K=\langle \varphi \rangle \cong {\mathbb Z}_{p}$. 
Then, $(S=X/K,N=H/K)$ is a ${\mathbb Z}_{p}^{d}$-action of type $(d;p;d+1)$.

\end{enumerate}
\end{example}

\section{${\mathbb Z}_{p}^{m}$-actions of type $(d;p,n)$, $d \geq 2$}
In this section, we assume $d \geq 2$.

\subsection{${\mathbb Z}_{p}^{m}$-actions as quotients of generalized Fermat varieties}
Let us consider a ${\mathbb Z}_{p}^{m}$-action $(S,N)$ of type $(d;p,n)$, and let $A={\rm Aut}(S)$ be the group of holomorphic automorphisms of $S$. 

Let us consider a Galois branched cover $\pi_{N}:S \to {\mathbb P}^{d}$ with deck group $N \cong {\mathbb Z}_{p}^{m}$ and whose branch locus consists of $(n+1)$ hyperplanes in general position. Up to postcomposition with a suitable element of ${\rm PGL}_{d+1}({\mathbb C})$, we may assume this $(n+1)$ hyperplanes to be given by the collection ${\mathcal B}(\Lambda)$, for a suitable $\Lambda \in \Omega_{n,d}$. 

As generalized Fermat varieties of type $(d;p,n)$ are universal (branched) covers of orbifolds with underlying space ${\mathbb P}^{d}$ and branch locus consisting of $(n+1)$ hyperplanes in general position (each one of cone order $p$), we may observe the following fact.

\begin{theo}
There is a subgroup ${\mathbb Z}_{p}^{n-m} \cong K \lhd H$, acting freely on $X_{n}^{p}(\Lambda)$, and a biholomorphism $\phi:S \to X_{n}^{p}(\Lambda)/K$ such that $\phi N \phi^{-1}=H/K$. In particular, (i) $m \leq n$, and (ii) if $m=n$, then $K=\{1\}$.
\end{theo}

As a consequence of the above, we will assume (and this will be in what follows) that $m \leq n-1$.

Let us denote by $\pi_{K}:X_{p}^{n}(\Lambda) \to S$ a Galois covering with deck group $K$.
The fact that $X_{p}^{n}(\Lambda)$ is simply connected ensures that $A$ lifts, under $\pi_{K}$, to a group $Q$ of biholomorphisms of $X_{p}^{n}(\Lambda)$, i.e., there is a short exact sequence
\begin{equation}\label{short1}
1 \to K \to Q \stackrel{\rho}{\to} A \to 1,
\end{equation}
where $\pi_{K} \circ \psi =\rho(\psi) \circ \pi_{K}$.

As $H/K=N \leq A$, it follows that $H \leq Q$. So, if $(d;p,n) \notin \{(2;2,5),(2;4,3)\}$, then the uniqueness of $H$ ensures that $H \lhd Q$, i.e., $N \lhd A$. In particular, the above short exact sequence determines (i) a short exact sequence
\begin{equation}\label{short3}
1 \to N \to A \stackrel{\theta}{\to} L \to 1,
\end{equation}
where $\pi_{N} \circ \psi =\theta(\psi) \circ \pi_{N}$, $L=A/N=Q/H$ is a subgroup of the ${\rm PGL}_{d+1}$-stabilizer of the configuration ${\mathcal B}(\Lambda)$, and (ii) a short exact sequence
\begin{equation}\label{short2}
1 \to H \to Q \stackrel{\eta}{\to} L \to 1,
\end{equation}
where $\pi \circ \psi =\eta(\psi) \circ \pi$. 

\begin{rema}
In particular, if $(p,|L|)=1$, then (by the Schur-Zassenhaus theorem), $Q \cong H \rtimes L$ and $A \cong K \rtimes L$.
\end{rema}

We have proved the following.

\begin{theo}
Let $(S,N)$ be a ${\mathbb Z}_{p}^{m}$-action $(S,N)$ of type $(d;p,n) \notin \{(2;2,5),(2;4,3)\}$ and $d \geq 2$. Then
\begin{enumerate}
\item $N \lhd {\rm Aut}(S)$.
\item Let $\pi:S \to {\mathbb P}^{d}$ be a Galois branched cover with deck group $N$ and with branch locus ${\mathcal B}$ being a collection of $n+1$ hyperplanes in general position.Then, there is a short exact sequence
\begin{equation}
1 \to N \to {\rm Aut}(S) \stackrel{\theta}{\to} L \to 1,
\end{equation}
where $\pi \circ \psi =\theta(\psi) \circ \pi$, and $L$ is a subgroup of the ${\rm PGL}_{d+1}$-stabilizer of ${\mathcal B}$.
\end{enumerate}
\end{theo}

\subsection{Uniqueness}
As already noticed, a generalized Fermat variety of type $(d;p,n) \notin \{(2;2,5),(2;4,3)\}$ admits a unique generalized Fermat group. The following result states a similar uniqueness result for 
${\mathbb Z}_{p}^{m}$-action $(S,N)$ of type $(d;p,n) \notin \{(2;2,5),(2;4,3)\}$ and $d \geq 2$.

\begin{theo} Let $d \geq 2$ and 
$(S,N)$ be a ${\mathbb Z}_{p}^{m}$-action $(S,N)$ of type $(d;p,n) \notin \{(2;2,5),(2;4,3)\}$. If $(S,M)$ is a ${\mathbb Z}_{q}^{r}$-action of type $(d;q,s)$, then $M=N$.
\end{theo}
\begin{proof}
Assume $S=X_{n}^{p}(\Lambda)/K$. Let $\hat{\psi} \in M \cong {\mathbb Z}_{q}^{r}$ be such that its locus of fixed points has dimension $d-1$. Let us consider a lifting $\psi \in {\rm Aut}(X_{n}^{p}(\Lambda))$ of $\hat{\psi}$. We may take $\psi$ so that its locus of fixed points has dimension $d-1$, so $\psi \in H$ is a non-trivial power of some canonical generator. So, $M \leq N$. Now, by looking at the equations for $H$ and $X_{n}^{p}$, we may observe that the only subgroup $L$ of $N$, for which $(S,L)$ is a ${\mathbb Z}_{p}^{r}$-action, is for $L=N$.
\end{proof}

\section{Freely acting subgroups of $H$}
As previously seen, if $(S,N)$ is a ${\mathbb Z}_{p}^{m}$-action of type $(d;p,n)$, then $(S,N)$ is biholomorphically equivalent to $(X_{n}^{p}(\Lambda)/K,H/K)$, where $\Lambda \in \Omega_{n,d}$ and $K$ is a subgroup of $H$ acting freely on $X_{n}^{p}(\Lambda)$ such that  $H/K \cong {\mathbb Z}_{p}^{m}$. The freely acting condition for $K$ is, by Proposition \ref{puntosfijos}, independent of the choice of $\Lambda$.

Let us denote by ${\mathcal F}(d;p,n,m)$ the collection of the subgroups $K$ of $H$ such that:
\begin{enumerate}
\item  $H/K \cong {\mathbb Z}_{p}^{m}$, and 
\item $K$ does not contain those  
$\varphi_{i_{1}}^{l_{1}} \varphi_{i_{2}}^{l_{2}} \cdots \varphi_{i_{j}}^{l_{j}}$, where $1 \leq j \leq d$, $l_{j} \in \{1,\ldots, p-1\}$ and $1 \leq i_{1}<\cdots < i_{j} \leq n+1$.
\end{enumerate}

Observe that this collection is invariant under the action of 
${\rm Aut}_{g}(H)$.

\begin{lemm}\label{emes}
If $d \geq 2$ and ${\mathcal F}(d;p,n,m) \neq \emptyset$, then $d \leq m$. Moreover, 
if $m=d=2$, then $p \geq 4$.
\end{lemm}
\begin{proof}
Let $\theta:H \to {\mathbb Z}_{p}^{m}$ be a surjective homomorphism such that $\ker(\theta)=K \in {\mathcal F}(d;p,n,m)$. Let us set $\theta(\varphi_{j})=\phi_{j}$.
As ${\rm Aut}_{g}(H)$ keeps invariant ${\mathcal F}(d;p,n,m)$, up to precomposition of $\theta$ by a suitable element of ${\rm Aut}_{g}(H)$, we may assume that $\theta(H)=\langle \phi_{1},\ldots,\phi_{m}\rangle$. 

As $\varphi_{1} \circ \cdots \circ \varphi_{n+1}=1$, we may observe that 
$$K=\langle \varphi_{1}^{l_{m+1,1}} \circ \cdots \circ \varphi_{m}^{l_{m+1,m}} \varphi_{m+1}^{-1}, \ldots,
\varphi_{1}^{l_{n,1}} \circ \cdots \circ \varphi_{m}^{l_{n,m}} \varphi_{n}^{-1} \rangle.$$ 
So, if $m<d$, then $K$ has elements of $H$ acting with fixed points, a contradiction.

Let us now assume $m=d=2$, $p \in \{2,3\}$, and that there is a surjective homomorphism $\theta:H \to {\mathbb Z}_{p}^{2}$ such that $\varphi_{k}, \varphi_{i}\varphi_{j}^{l} \notin K=\ker(\theta)$, for $l \in \{1,\ldots,p-1\}$. In particular, $\langle \theta(\varphi_{1})=\phi_{1}, \theta(\varphi_{2})=\phi_{2}\rangle = {\mathbb Z}_{p}^{2}$. 
For $j=3,\ldots,n+1$, $\theta(\varphi_{j})=\phi_{1}^{r_{j}}\phi_{2}^{s_{j}}$, where $r_{j},s_{j} \in \{0,\ldots,p-1\}$. Since $\varphi_{j}, \varphi_{1}\varphi_{j},\varphi_{2}\varphi_{j}, \varphi_{1}\varphi_{j}^{p-1},\varphi_{2}\varphi_{j}^{p-1} \notin K$, then $r_{j}=s_{j} \in \{1,2\}$. But, in this situation $\varphi_{3}\varphi_{4}$ or $\varphi_{3}\varphi_{4}^{2} \in K$, a contradiction.
\end{proof}

\subsubsection{\bf Description of elements of ${\mathcal F}(2;p,n,m)$}
Let $K \in {\mathcal F}(2;p,n,m)$. By the definition of ${\mathcal F}(2;p,n,m)$, 
$K$ does not contain those non-trivial elements of the form 
$\varphi_{k}, \varphi_{i} \varphi_{j}^{l}$, where $1 \leq k \leq n+1$, $1 \leq i < j \leq n+1$, and $l \in \{1,\ldots, p-1\}$.

Let us consider a surjective homomorphism $\theta_{1}:H \to {\mathbb Z}_{p}^{m}$ whose kernel is $K$. 
There is a subset (not unique) of indices $1=i_{1}<i_{2}<\cdots<i_{m} \leq n+1$ such that $\langle \phi_{1}=\theta_{1}(\varphi_{i_{1}}),\ldots,\phi_{m}=\theta_{1}(\varphi_{i_{m}})\rangle ={\mathbb Z}_{p}^{m}$. Let $\Phi \in {\rm Aut}_{g}(H)$ be such that $\Phi^{-1}(\varphi_{j})=\varphi_{i_{j}}$, for $j=1,\ldots,m$. Then $\Phi(K) \in {\mathcal F}(2;p,n,m)$ is the kernel of the surjective homomorphism $\theta=\theta_{1} \circ \Phi^{-1}:H \to {\mathbb Z}_{p}^{m}$. Note that 
$$\theta(\varphi_{j})=\phi_{j}, \; j=1,\ldots,m,$$
$$\theta(\varphi_{i})=\phi_{1}^{r_{i,1}} \cdots \phi_{m}^{r_{i,m}}, \; i=m+1,\ldots, n+1,$$
where the tuples $(r_{i,1},\ldots,r_{i,m}) \in \{0,1,\ldots,p-1\}^{m}$ satisfy the following properties.
\begin{enumerate}
\item ($\varphi_{1} \cdots \varphi_{n+1}=1$)
$$1+r_{m+1,i}+ r_{m+2,i} +\cdots+r_{n+1,i} \equiv 0 \mod(p), \; i=1,\ldots,m.$$

\item ($\varphi_{i}\notin K$, for $i=m+1,\ldots,n+1$)
$$(r_{i,1},\ldots,r_{i,m})  \neq (0,\ldots,0), \; i=m+1,\ldots,n+1.$$

\item ($\varphi_{k} \varphi_{i}^{l} \notin K$, for $k=1,\ldots,m$, $i=m+1,\ldots,n+1$, and $l=1,\ldots,p-1$)

$(r_{i,1},\ldots,r_{i,m})$ cannot have $(m-1)$ of its coordinates equal to zero, for $i=m+1,\ldots,n+1.$

\item ($\varphi_{i} \varphi_{j}^{l} \notin K$, for $m+1 \leq i < j \leq n+1$, and $l=1,\ldots,p-1$)
$$(r_{i,1}+lr_{j,1},\ldots,r_{i,m}+lr_{j,m})  \nequiv (0,\ldots,0) \mod(p), \; m+1 \leq i < j \leq n+1, \; l=1,\ldots,p-1.$$
\end{enumerate}

In this case, 
$$\Phi(K)=\langle \varphi_{1}^{r_{m+1,1}}\cdots \varphi_{m}^{r_{m+1,m}}\varphi_{m+1}^{-1}, \ldots,  \varphi_{1}^{r_{n,1}}\cdots \varphi_{m}^{r_{n,m}}\varphi_{n}^{-1} \rangle.$$

Summarizing the above is the following.

\begin{theo}
Up to ${\rm Aut}_{g}(H)$, the elements of ${\mathcal F}(2;p,n,m)$ are given by the following normalized ones
$$K=\langle \varphi_{1}^{r_{m+1,1}}\cdots \varphi_{m}^{r_{m+1,m}}\varphi_{m+1}^{-1}, \ldots,  \varphi_{1}^{r_{n,1}}\cdots \varphi_{m}^{r_{n,m}}\varphi_{n}^{-1} \rangle,$$
where the exponents $r_{i,j} \in \{0,1,\ldots,p-1\}$ satisfy the conditions (1)-(4) as described above.
\end{theo}

\subsubsection{The case $d=p=2$}
As already noticed in Lemma \ref{emes}, in this case $m \geq 3$. In the following, we observe that, for $m=3$, necessarily $n=6$.

\begin{prop}\label{propo3}\mbox{}
\begin{enumerate}
\item ${\mathcal F}(2;2,n,3) \neq \emptyset$ if and only if $n=6$.
Moreover, ${\mathcal F}(2;2,6,3)/{\rm Aut}_{g}(H)$ has exactly one element, this one represented by the group $K=\langle \varphi_{1}\varphi_{2}\varphi_{4}, \varphi_{1}\varphi_{3}\varphi_{5},\varphi_{2}\varphi_{3}\varphi_{6}\rangle$.
\item ${\mathcal F}(2;2,n,n-1) \neq \emptyset$, for $n \geq 5$.
\item ${\mathcal F}(2;2,n,n-2) \neq \emptyset$, for $n \geq 6$.
\item ${\mathcal F}(2;2,(m-1)(m+2)/2,m) \neq \emptyset$, for $m \geq 4$ even.
\item ${\mathcal F}(2;2,m(m+1)/2,m) \neq \emptyset$, for $m \geq 3$ odd.
\end{enumerate}
\end{prop}
\begin{proof}
Part (1): we may check by direct inspection that
${\mathcal F}(2;2,4,3)={\mathcal F}(2;2,5,3)=\emptyset$.
Assume ${\mathcal F}(2;2,n,3) \neq \emptyset$, where $n \geq 6$. Up to ${\rm Aut}_{g}(H)$, there is a surjective homomorphism $\theta:H \to {\mathbb Z}_{2}^{3}=\langle \phi_{1},\phi_{2},\phi_{3}\rangle$,
where $\phi_{j}=\theta(\varphi_{j})$, for $j=1,2,3$, and 
$\varphi_{k}, \varphi_{i} \varphi_{j} \notin K=\ker(\theta)$, where $1 \leq k \leq n+1$, and $1 \leq i < j \leq n+1$. 
Let us write, for $j=4,\ldots,n+1$,  $\theta(\varphi_{j})=\phi_{1}^{r_{j}}\phi_{2}^{s_{j}}\phi_{3}^{t_{j}}$, where $r_{j},s_{j},t_{j} \in \{0,1\}$. The condition that $\varphi_{j} \notin K$ is equivalent to have that $(r_{j},s_{j},t_{j}) \neq (0,0,0)$. The condition that $\varphi_{i}\varphi_{j} \notin K$, for $i \in \{1,2,3\}$ and $j\in\{4,\ldots,n+1\}$, is equivalent to have that $(r_{j},s_{j},t_{j}) \neq (1,0,0), (0,1,0),(0,0,1)$. In particular, $(r_{j},s_{j},t_{j}) \in \{(1,1,1), (1,1,0),(0,1,1),(1,0,1)\}$. The condition that $\varphi_{i}\varphi_{j} \notin K$, for $4 \leq i <j \leq n+1$ is equivalent to have that for different indices $4 \leq i<j \leq n+1$, $(r_{i},s_{i},t_{i}) \neq (r_{j},s_{j},t_{j})$. This ensures that $n=6$ and that, up to ${\rm Aut}_{g}(H)$, we may choose $(r_{4},s_{4},t_{4})=(1,1,0)$, $(r_{5},s_{5},t_{5})=(1,0,1)$, $(r_{6},s_{6},t_{6})=(0,1,1)$, and  $(r_{7},s_{7},t_{7})=(1,1,1)$.

Part (2): just consider the surjective homomorphism
$\theta:H \to {\mathbb Z}_{2}^{n-1}=\langle \phi_{1},\ldots,\phi_{n-1}\rangle$, defined by 
$\theta(\varphi_{k})=\phi_{k}, \; k=1,\ldots n-1,$
$\theta(\varphi_{n})=\phi_{i_{1}}\cdots \phi_{i_{l_{1}}},$ and 
$\theta(\varphi_{n+1})=\phi_{j_{1}}\cdots \phi_{j_{l_{2}}},$
where $\{i_{1},\ldots,i_{l_{1}}\}$ and $\{j_{1},\ldots,j_{l_{2}}\}$ is a disjoint partition of $\{1,\ldots,n-1\}$, with $l_{1},l_{2} \geq 2$.

Part (3): just consider the surjective homomorphism
$\theta:H \to {\mathbb Z}_{2}^{n-2}=\langle \phi_{1},\ldots,\phi_{n-2}\rangle$, defined by 
$\theta(\varphi_{k})=\phi_{k}, \; k=1,\ldots n-2,$
$\theta(\varphi_{n-1})=\phi_{i_{1}}\cdots \phi_{i_{l_{1}}},$
$\theta(\varphi_{n})=\phi_{j_{1}}\cdots \phi_{j_{l_{2}}}$ and
$\theta(\varphi_{n+1})=\phi_{k_{1}}\cdots \phi_{k_{l_{3}}},$
where $\{i_{1},\ldots,i_{l_{1}}\}$, $\{j_{1},\ldots,j_{l_{2}}\}$,
and $\{j_{1},\ldots,j_{l_{3}}\}$ is a disjoint partition of $\{1,\ldots,n-2\}$, with $l_{j} \geq 2$.

Part (4): just consider the surjective homomorphism
$\theta:H \to {\mathbb Z}_{2}^{m}=\langle \phi_{1},\ldots,\phi_{m}\rangle$, defined by 
$\theta(\varphi_{k})=\phi_{k}, \; k=1,\ldots m,$
and $\{a_{m+1},\ldots,n+1\}$ are sent to $\{\phi_{1}\phi_{2},\ldots,\phi_{m-1}\phi_{m}\}$ bijectively.

Part (5): just consider the surjective homomorphism
$\theta:H \to {\mathbb Z}_{2}^{m}=\langle \phi_{1},\ldots,\phi_{m}\rangle$, defined by 
$\theta(\varphi_{k})=\phi_{k}, \; k=1,\ldots m,$
and $\{a_{m+1},\ldots,n\}$ are sent to $\{\phi_{1}\phi_{2},\ldots,\phi_{m-1}\phi_{m}\}$ bijectively, and $\theta(\varphi_{n+1})=\phi_{1}\cdots\phi_{m}$.
\end{proof}

\begin{example}
By Proposition \ref{propo3}, for the type ${\mathcal F}(2;2,6,3)/{\rm Aut}_{g}(H)$ has cardinality one. A representative is 
$$K=\langle \varphi_{1}\varphi_{2}\varphi_{4}, \varphi_{1}\varphi_{3}\varphi_{5},\varphi_{2}\varphi_{3}\varphi_{6}\rangle.$$

This provides the $6$-dimensional family 
$$\left\{\left(S_{\Lambda}=X_{6}^{2}(\Lambda)/K, N_{\Lambda}=H/K\right): \Lambda \in \Omega_{6,2}\right\}$$
of ${\mathbb Z}_{2}^{3}$-actions of type $(2;2,6,3)$, all of them topologically conjugated.
Below, we proceed to compute algebraic equations for these pairs $(S_{\Lambda},N_{\Lambda})$.

Let us first consider the affine model $X(\Lambda) \subset {\mathbb C}^{6}$ of $X_{6}^{2}(\Lambda)$ by taking $x_{7}=1$. In this affine model, $K$ is generated by the linear transformations
$$\eta_{1}(x_{1},\ldots,x_{6})=(-x_{1},-x_{2},x_{3},-x_{4},x_{5},x_{6}),$$
$$\eta_{2}(x_{1},\ldots,x_{6})=(-x_{1},x_{2},-x_{3},x_{4},-x_{5},x_{6}),$$
$$\eta_{3}(x_{1},\ldots,x_{6})=(x_{1},-x_{2},-x_{3},x_{4},x_{5},-x_{6}).$$

A set of generators for the invariants ${\mathbb C}[x_{1},\ldots,x_{6}]^{K}$ is
$$u_{1}=x_{1}^{2}, u_{2}=x_{2}^{2}, u_{3}=x_{3}^{2}, u_{4}=x_{4}^{2}, u_{5}=x_{5}^{2}, u_{6}=x_{6}^{2}, u_{7}=x_{1}x_{2}x_{3}, u_{8}=x_{1}x_{4}x_{5},$$
$$u_{9}=x_{2}x_{4}x_{6}, u_{10}=x_{3}x_{5}x_{6}, u_{11}=x_{1}x_{2}x_{5}x_{6}, u_{12}=x_{1}x_{3}x_{4}x_{6}, u_{13}=x_{2}x_{3}x_{4}x_{5}.$$

So, if we consider the map $\Phi:{\mathbb C}^{6} \to {\mathbb C}^{13}$, defined by $\Phi(x_{1},\ldots,x_{6})=(u_{1},\ldots,u_{13})$, then $\Phi(X(\Lambda))$ is isomorphic to the affine model of $S_{\Lambda}$. The image (affine) surface $\Phi(X(\Lambda))$ is defined by the following equalities

$$u_{6}u_{13} =u_{9}u_{10},
u_{5}u_{12} = u_{8}u_{10},
u_{1}u_{2}u_{3} = u_{7}^{2},
u_{5}u_{6}u_{7} = u_{10}u_{11},
u_{4}u_{11} = u_{8}u_{9},
u_{1}u_{2}u_{5}u_{6} = u_{11}^{2},
$$
$$
u_{4}u_{6}u_{7 } =u_{9}u_{12},
u_{1}u_{2}u_{10} = u_{7}u_{11},
u_{4}u_{5}u_{7 }= u_{8}u_{13},
u_{3}u_{11} = u_{7}u_{10},
u_{1}u_{3}u_{4}u_{6} = u_{12}^{2},
u_{3}u_{6}u_{8} = u_{10}u_{12},
$$
$$
u_{3}u_{5}u_{9} = u_{10}u_{13},
u_{3}u_{5}u_{6} = u_{10}^{2},
u_{3}u_{8}u_{9} = u_{12}u_{13},
u_{2}u_{12} = u_{7}u_{9},
u_{1}u_{3}u_{9} = u_{7}u_{12},
u_{2}u_{6}u_{8} = u_{9}u_{11},
$$
$$
u_{2}u_{8}u_{10} = u_{11}u_{13},
u_{2}u_{4}u_{10} = u_{9}u_{13,}
u_{2}u_{4}u_{6} = u_{9}^{2},
u_{1}u_{4}u_{5} = u_{8}^{2},
u_{2}u_{3}u_{8} = u_{7}u_{13},
u_{2}u_{3}u_{4}u_{5} = u_{13}^{2},
$$
$$
u_{1}u_{13} = u_{7}u_{8},
u_{1}u_{4}u_{10} = u_{8}u_{12},
u_{1}u_{5}u_{9} = u_{8}u_{11},
u_{1}u_{9}u_{10} = u_{11}u_{12}
$$
$$
u_{4}=-u_{1}-u_{2}-u_{3},
u_{5}=-\lambda_{1,1}u_{1}-\lambda_{1,2}u_{2}-u_{3},
u_{6}=-\lambda_{2,1}u_{1}-\lambda_{2,2}u_{2}-u_{3},
u_{3}=-\lambda_{3,1}u_{1}-\lambda_{2,2}u_{2}-1.
$$

In this model, the group $N=\langle \phi_{1},\phi_{2},\phi_{3}\rangle$ is given by:
$$\phi_{1}: \left\{\begin{array}{ll}
u_{i} \mapsto -u_{i}, & i=7,8,11,12\\
u_{j} \mapsto u_{j}, & \mbox{otherwise}
\end{array}\right.
$$
$$\phi_{2}: \left\{\begin{array}{ll}
u_{i} \mapsto -u_{i}, & i=7,9,11,13\\
u_{j} \mapsto u_{j}, & \mbox{otherwise}
\end{array}\right.
$$
$$\phi_{3}: \left\{\begin{array}{ll}
u_{i} \mapsto -u_{i}, & i=7,10,12,13\\
u_{j} \mapsto u_{j}, & \mbox{otherwise}
\end{array}\right.
$$

\end{example}

\subsection{On topologically equivalence}
Two ${\mathbb Z}_{p}^{m}$-actions $(S_{1},N_{1})$ and $(S_{2},N_{2})$, both of type $(d;p,n)$, are topologically equivalent if there is an orientation-preserving homeomorphism $F:S_{1} \to S_{2}$ such that $F N_{1} F^{-1}=N_{2}$.  
Assume that
$S_{j}=X_{n}^{p}(\Lambda_{j})/K_{j}$, and $N_{j}=H/K_{j}$, where $\Lambda_{j} \in \Omega_{n,d}$ and $K_{j} \in {\mathcal F}(d;p,n,m)$. Then, as $X_{n}^{p}(\Lambda_{j})$ are universal covers, $F$ lifts to an orientation-preserving homeomorphism $\widetilde{F}:X_{n}^{p}(\Lambda_{1}) \to X_{n}^{p}(\Lambda_{2})$ such that $\widetilde{F} K_{1} \widetilde{F}^{-1}=K_{2}$. The homomorhism $\widetilde{F}$ induces, by the conjugation action, an element $\Phi \in {\rm Aut}_{g}(H)$, which satisfies that $\Phi(K_{1})=K_{2}$. We have obtained the following fact.

\begin{prop}
If $K_{1}, K_{2} \in {\mathcal F}(d;p,n,m)$ determine topologically equivalent ${\mathbb Z}_{p}^{m}$-actions of type $(d;p,n)$, then there exists some $\Phi \in {\rm Aut}_{g}(H)$ such that
$K_{2}=\Phi(K_{1})$.
\end{prop}

Now, assume that we have 
$K_{1}, K_{2} \in {\mathcal F}(d;p,n,m)$ such that there is some $\Phi \in {\rm Aut}_{g}(H)$ satisfying 
$K_{2}=\Phi(K_{1})$. Is such $\Phi$ induced by an orientation-preserving homeomorphism? If this is the case, then the above result will state that the number of 
topologically equivalent ${\mathbb Z}_{p}^{m}$-actions of type $(d;p,n)$ is equal to the cardinality of ${\mathcal F}(d;p,n,m)/{\rm Aut}_{g}(H)$.
This is true for $d=1$ \cite{HR}, but it is not clear for $d \geq 2$.


\section{On hyperbolicity of ${\mathbb Z}_{p}^{m}$-actions}
Let $S$ be a compact complex manifold of dimension $d \geq 2$.
The manifold $S$ is Kobayashi hyperbolic if its Kobayashi pseudometric is non-degenerate. In \cite{Brody}, Brody observed that $S$ is Kobayashi hyperbolic if and only if 
there is no non-constant holomorphic map $f:{\mathbb C} \to S$.

Assume that $S$ is a projective variety. In \cite{Demailly}, Demailly introduced an algebraic analogue for hyperbolicity. More precisely, $S$ is called 
algebraically hyperbolic if there exists a positive constant $A$ such that the degree of any curve of genus $g$ on $S$ is bounded from above by $A(g-1)$. 
In the same paper, Demailly proved that Kobayashi hyperbolicity implies algebraically hyperbolicity.  By the definition, an algebraically hyperbolic manifold does not contain genus $g \in \{0,1\}$ curves.

In \cite{BKV}, Bogomolov, Kamenova, and Verbitsky proved that, if $S$ is algebraically hyperbolic, then  ${\rm Aut}(S)$ is finite (for the Kobayashi hyperbolic case, this was proved by Kobayashi in \cite{Ko}).

\medskip

Let us consider a ${\mathbb Z}_{p}^{m}$-action $(S,N)$ of type $(d;p,n)$, where $n \geq d+1$.

\subsection{Case $m=n$ and $(d;p,n) \in \{(2;4,3),(2;2,5)\}$}
If $(d;k,n)=(2;4,3)$, then $S$ corresponds to the classical Fermat hypersurface of degree $4$ in ${\mathbb P}^{3}$ for which ${\rm Lin}(S) \cong {\mathbb Z}_{4}^{3} \rtimes {\mathfrak S}_{4}$ and ${\rm Aut}(S)$ infinite; so $S$ is not algebraically hyperbolic.
If $(d;k,n)=(2;2,5)$, then ${\rm Lin}(S)$ is a finite extension of ${\mathbb Z}_{2}^{5}$ (generically a trivial extension) and ${\rm Aut}(S)$ is infinite by results due to Shioda and Inose in \cite[Thm 5]{Shioda} (in \cite{Vinberg} Vinberg computed it for a particular case). So, again, these surfaces are not algebraically hyperbolic.

\subsection{Case $m=n$ and $(d;p,n) \notin \{(2;4,3),(2;2,5)\}$}
Let us now asume that $(d;p,n) \notin \{(2;4,3),(2;2,5)\}$, where $n \geq d+1$. In this case, we know that $S$ is a compact projective complex manifold of dimension $d$ with ${\rm Aut}(S)$ finite. 
We wonder if, in these cases, $S$ is or is not algebraically hyperbolic.

\subsection{Case $d+1 \leq m \leq n \leq 2d-1$}
In the next result, we observe that, for $n\leq 2d-1$, $S$ cannot be algebraically hyperbolic.

\begin{theo}
If  $(S,N)$ is a ${\mathbb Z}_{p}^{m}$-action of type $(d;p,n)$, where $3 \leq d+1 \leq n$. Then, in the following situations, $S$ is not algebraically hyperbolic.
\begin{enumerate}
\item $n \leq 2d-1$.
\item $n=2d$ and $p \in \{2,3\}$.
\item $n=2d+1$ and $p=2$.
\end{enumerate}
\end{theo}
\begin{proof}
Let $\pi_{N}:S \to {\mathbb P}^{d}$ be a Galois branched covering with deck group $N$, whose branch locus is given by the collection 
${\mathcal B}$, consisiting of the $n+1$ hyperplanes 
$\Sigma_{1}, \ldots,\Sigma_{n+1}$, that are in general position.
By the general position condition, the intersection of the planes $\Sigma_{1},\ldots,\Sigma_{d}$ consists of a unique point 
$\alpha$. 

(1) Let us first consider the case $n \leq 2d-1$.
Now, let us consider the intersection of the $n+1-d$ hyperplanes $\Sigma_{d+1},\ldots,\Sigma_{n+1}$, which is non-empty since $n+1-d \leq d$. Again, by the general position condition, we can find a point $\beta$ in that intersection that does not belong to $\Sigma_{j}$, for $j=1,\ldots,d$.
 Let $L \subset {\mathbb P}^{d}$ the line connecting $\alpha$ with $\beta$. We observe that $L \cap {\mathcal B}(\Lambda)=\{\alpha,\beta\}$. Set $L^{*}=L\setminus\{\alpha,\beta\} \cong {\mathbb C}\setminus\{0\}$. Let $\hat{L}$ be any connected component of $\pi_{N}^{-1}(L^{*})$, which is a Riemann surface that finitely covers $L^{*}$. In this way, inside $S$ we have a genus zero curve (by adding the two missing points to $\hat{L}$), so $S$ cannot be algebraically hyperbolic.
 
(2) Let us now assume that $n=2d$. We proceed similarly as in the previous case, but in this case, we consider the intersection of the $d$ hyperplanes $\Sigma_{d+1},\ldots,\Sigma_{2d}$; which is a point $\beta$. We consider the line $L \subset {\mathbb P}^{d}$ connecting $\alpha$ and $\beta$. In this case, $L$ intersects $\Sigma_{2d+1}$ in a third point $\gamma$. 
Set $L^{*}=L\setminus\{\alpha,\beta,\gamma\} \cong {\mathbb C}\setminus\{0,1\}$. Let $\hat{L}$ be any connected component of $\pi_{N}^{-1}(L^{*})$, which is a punctured Riemann surface. Moreover, $\pi_{N}:\hat{L} \to L^{*}$ is a finite abelian cover of degree $p^{2}$. By adding the missing punctures to $\hat{L}$, we obtain a closed Riemann surface $W$ such that $\pi_{N}:W \to L$ is an abelian covering, with three branch values, each of order $p$. By the Riemann-Hurwitz formula, if $p \in \{2,3\}$, then $W$ has genus $0$ or $1$.
 So, $S$ cannot be algebraically hyperbolic. 
 
 (3) The argument is similar to that in case (2), except that in this case $L$ intersects the branch locus of $\pi_{N}$ in four points. So, we will have an abelian covering $W \to L$, branched at four points, each of order $2$. This again ensures that $W$ has genus one.
\end{proof}

\begin{example}
Let us consider a generalized Fermat variety $X=X_{4}^{2}(\Lambda)$ of type $(2;2,4)$; so $n=2d$ and we are in case (2) of the previous result. In this case, the locus of fixed points $F_{1} \subset X$ of $\varphi_{1}$ has genus one, in particular, $X$ is not algebraically hyperbolic.
\end{example}

\begin{question}
Let $(S,N)$ be a ${\mathbb Z}_{p}^{m}$-action of type $(d;p,n)$, where $d \geq 2$, $n \geq 2d$ and, if $n=2d$, then $p \geq 4$, and if $n=2d+1$, then $p \geq 3$. When is $S$ algebraically hyperbolic?
\end{question}

\section*{\bf Acknowledgements}
The first author would like to thank the {\it Instituto de Matem\'atica} at Universidad de Talca for providing both a challenging and a motivating environment during a visiting position from September to November 2025, where this project started. In particular, to thank Max Leyton and Alvaro Liendo for the fruitful conversations concerning this (and other) mathematical ideas.







\begin{thebibliography}{99}
\bibitem{AlPa13}
Alexeev, V. and Pardini, R.
On the existence of ramified abelian covers.
{\it Rend. Semin. Mat. Univ. Politec. Torino} {\bf 71} (2013), 307--315.



\bibitem{BM}
Bochner, S. and Montgomery, D.
Groups on analytic manifolds. 
{\it Ann. of Math.} {\bf 48} (1947), 659--669.

\bibitem{BKV}
Bogomolov, F., Kamenova, L., and Verbitsky, M.
Algebraically hyperbolic manifolds have finite automorphism groups.
{\it Communications in Contemporary Mathematics} {\bf 22}, 1950003 (2020).
 
 
\bibitem{Brody}
Brody, R.
Complex manifolds in hyperbolicity.
{\it Trans. Amer. Math. Soc.} {\bf 235} (1978), 213--219.

\bibitem{Demailly}
Demailly, J.-P. 
Algebraic criteria for Kobayashi hyperbolic projective varieties and jet diﬀerentials. 
{\it Proc. Symp. Pure Math.} {\bf 62} (1997), 285--360.
 
\bibitem{Dummit}
Dummit, D. S., and R. M. Foote, R. M.
{\it Abstract Algebra}. Wiley (2003). ISBN 978-0-471-43334-7.

 
 
 
 



\bibitem{GHL09}
Gonz\'{a}lez-Diez, G., Hidalgo, R. A. and Leyton, M.
Generalized {F}ermat curves.
{\it Journal of Algebra} {\bf 321} (2009), 1643--1660.


\bibitem{Har74}
R. Hartshorne.
Varieties of small codimension in projective space.
{\it Bull. Amer. Math. Soc.}, 80:1017--1032, 1974.
     
\bibitem{Har77}
Hartshorne, R.
{\it Algebraic geometry}.
Graduate Texts in Mathematics {\bf 52}, Springer-Verlag, New York, 1977.


\bibitem{HHL23}
Hidalgo, R. A., Hughes, H. F. and Leyton-Alvarez, M.
Uniqueness of Generalized Fermat Groups in Positive Characteristic.
{\it Transformation Groups} (2025). https://doi.org/10.1007/s00031-025-09910-6

\bibitem{HKLP17}
Hidalgo, R. A., Kontogeorgis, A., Leyton-\'Alvarez, M. and Paramantzoglou, P.
Automorphisms of generalized {F}ermat curves.
{\it Journal of Pure and Applied Algebra} {\bf 221} (2017), 2312--2337.

\bibitem{HR}
Hidalgo, R. A., and Reyes-Carocca, S.
${\mathbb Z}_{k}^{m}$-actions of signature $(0;k,\stackrel{n+1}{\dots},k)$. In preparation.


\bibitem{Ko}
Kobayashi, S.
Intrinsic distances, measures and geometric function theory. 
{\it Bull. Amer. Math. Soc.} {\bf 82} (1976), 357--416.


\bibitem{Kon02}
Kontogeorgis, A.
Automorphisms of {F}ermat-like varieties.
{\it Manuscripta Mathematica} {\bf 107} (2002), 187--205.

\bibitem{MaMo64}
Matsumura, H. and Monsky, P.
On the automorphisms of hypersurfaces.
{\it Journal of Mathematics of Kyoto University} {\bf 3} (1963/1964), 347--361.






\bibitem{Par91}
Pardini, R.
Abelian covers of algebraic varieties.
{\it Journal f\"{u}r die Reine und Angewandte Mathematik} {\bf 417} (1991), 191--213.




\bibitem{Randell}
Randell, R.
Lattice-isotopic arrangements are topologically isomorphic.
{\it Proc. Amer. Math. Soc.} {\bf 107} (1989), 555--559.



 
 
\bibitem{Shioda}
T. Shioda and H. Inose.
On singular K3 surfaces.
{\it In Complex analysis and algebraic geometry}, pages 119--136. Iwanami Shoten, Tokyo, 1977.

 
 
 





\bibitem{Vinberg}
Vinberg, \`E. B.
The two most algebraic $K3$ surfaces.
{\it Mathematische Annalen} {\bf 265} (1983), 1--21.

\end{thebibliography}
\end{document}